\documentclass{itogi20171} 

 \currentyear{2023}
 \currentvolume{???}

\usepackage[russian]{babel}
\usepackage[cp1251]{inputenc}
\usepackage{graphicx}
\usepackage{enumerate} 
\usepackage{cmap} 
\usepackage{comment}
\usepackage{mathdots}
\usepackage{wasysym} 
\usepackage{xcolor} 

\usepackage[
unicode,colorlinks,linkcolor=blue,citecolor=red,bookmarksopen,pdfhighlight=/N]{hyperref}

\newtheorem{thm}{Теорема}
\newtheorem*{thA}{Теорема A}
\newtheorem*{thB}{Теорема B}
\newtheorem{proposition}{Предложение}
\newtheorem{lemma}{Лемма}

\theoremstyle{definition}
\newtheorem{remark}{Замечание}
\newtheorem{example}{Пример}

\renewcommand{\leq}{\leqslant} 
\renewcommand{\geq}{\geqslant}
\newcommand{\RR}{\mathbb{R}} 
\newcommand{\CC}{\mathbb{C}} 
\newcommand{\NN}{\mathbb{N}} 
\newcommand{\ZZ}{\mathbb{Z}}

\DeclareMathOperator{\supp}{supp} 
 
\DeclareMathOperator{\conv}{\mathsf{conv}}
\DeclareMathOperator{\dd}{d\!}
\DeclareMathOperator{\intr}{\mathsf{int}}
\DeclareMathOperator{\Hol}{\mathsf{Hol}}
\DeclareMathOperator{\dens}{\mathsf{dens}}
\DeclareMathOperator{\Exp}{\mathsf{Exp}}
\DeclareMathOperator{\Spf}{\mathsf{Spf}}
\DeclareMathOperator{\prm}{\mathsf{prm}}
\DeclareMathOperator{\clos}{\mathsf{clos}}
\DeclareMathOperator{\rad}{rad}

\renewcommand{\Re}{\operatorname{Re}}

\begin{document}

\title[Полнота  экспоненциальных систем \dots]{Полнота  экспоненциальных систем\\  и периметр выпуклой оболочки}

\author[Хабибуллин Булат Нурмиевич]{Б. Н. Хабибуллин}
\address{Институт математики с вычислительным центром  Уфимского федерального исследовательского центра Российской академии наук}
\email{khabib-bulat@mail.ru}

\author[Кудашева Елена Геннадьевна]{Е. Г. Кудашева}
\address{Башкирский государственный педагогический университет им. М. Акмуллы}
\email{lena\_kudasheva@mail.ru}

\author[Мурясов Роман Русланович]{Р. Р. Мурясов}
\address{Уфимский университет науки и технологий}
\email{romrumur@yandex.ru}

\thanks{Работа выполнена в рамках государственного задания Министерства науки и высшего образования 
Российской Федерации (код научной темы FMRS-2022-0124) при поддержке   Министерства просвещения Российской Федерации в рамках государственного задания (соглашение № 073-03-2023-010 от 26.01.2023).}

\keywords{полнота систем функций, экспоненциальная система, целая функция экспоненциального типа, распределение корней, периметр, выпуклая оболочка, опорная функция}

\subjclass{30B60, 30D15,	52A38, 31A05}

\UDC{517.538.2, 517.547.22,  514.17, 517.574}

\begin{abstract}
Устанавливается новая шкала условий полноты экспоненциальных систем в двух видах  функциональных пространств  на подмножествах комплексной плоскости. Первый ---  банаховы пространства  функций, непрерывных на компакте и одновременно  голоморфных во внутренности этого компакта, если она непуста, с равномерной нормой. 
Второй --- пространства голоморфных функций на ограниченном открытом множестве с топологией равномерной сходимости на компактах. Эти условия формулируются в терминах мажорирования  периметра   выпуклой оболочки области определения функций из пространства  новыми  характеристиками распределения показателей экспоненциальной системы.   
\end{abstract}

\maketitle

\tableofcontents

\section{Введение}
\subsection{Некоторые обозначения, определения и соглашения}
Одноточечные множества $\{a\}$ часто записываем без фигурных скобок, т.е. просто как $a$. 
Так, для  множества $\mathbb N:=\{1,2, \dots\}$ всех {\it натуральных чисел\/} 
 $\NN_0:=0\bigcup \NN=\{0,1, 2, \dots\}$.
\textit{Множества всех действительных чисел\/} $\RR$ с таким же отношением порядка $\leq$  
рассматриваем и  как \textit{вещественную ось\/}  в \textit{комплексной плоскости\/} $\CC$ с евклидовой  нормой-модулем $|\cdot|$. Порядковое пополнение множества $\RR$ верхней гранью $+\infty  :=\sup \RR\notin \RR$ и нижней  гранью $-\infty  :=\inf \RR\notin \RR$ даёт \textit{расширенное\/} множество действительных чисел  $\overline \RR:= \RR\bigcup \{\pm\infty\}$ с порядковой топологией. \textit{Интервалы с концами  $a\in \overline \RR$ и $b\in \overline \RR$} --- это множества 
$[a,b]:=\bigl\{x\in \overline \RR\bigm| a\leq x\leq b\bigr\}$ --- \textit{отрезок\/} в $\overline \RR$, 
$(a,b]:=[a,b]\setminus a$, $[a,b):=[a,b]\setminus b$, а  
$(a,b):=[a,b)\setminus a$ --- \textit{открытый интервал\/} в $\overline \RR$. 
Используем также  обозначения $\RR^+:=[0, +\infty)$  
 для \textit{положительной полуоси\/} и   $\overline \RR^+:=[0, +\infty]$ для её расширения.
При  $r\in \overline \RR^+$ через $D(r):=\bigl\{z'\in \CC\bigm| |z|<r\bigr\}$ и 
$\overline D(r):=\bigl\{z'\in \CC\bigm| |z|\leq r\bigr\}$, а также 
$\partial \overline D(r):=\overline D(r)\setminus  D(r)$ 
 обозначаем соответственно \textit{открытый и замкнутый круги,}  а также \textit{окружность с центром в нуле  радиуса $r$.} Для подмножества $S\subset \CC$ через $\clos S$, $\intr S$, $\partial S$ и  $\conv S$
обозначаем соответственно  \textit{замыкание,\/} \textit{внутренность,\/} \textit{границу\/}  и \textit{выпуклую оболочку\/}  множества $S$ в  $\CC$. Если граница $\partial S$ подмножества $S\subset \CC$ --- спрямляемая замкнутая кривая, то  евклидову длину этой границы  $\partial S$ обозначаем через ${\prm} (\partial S)$ --- \textit{периметр границы\/   $\partial S$.}

Всюду далее через $Z$ обозначаем \textit{распределение точек\/} на комплексной плоскости\/ $\mathbb C$, среди которых  могут быть повторяющиеся. Распределение точек $Z$ однозначно определяется функцией, действующей из $\mathbb C$ в $\overline \NN_0$ и равной в каждой точке $z\in \mathbb C$ количеству  повторений этой точки $z$  в распределение точек $Z$. Для такой функции, которую часто называют \textit{функцией кратности,\/} или \textit{дивизором,\/} распределения точек  $Z$ \cite[пп.~0.1.2--0.1.3]{Khsur}, 
сохраняем то же обозначение $Z$. Другими словами, $Z(z)$ --- это количество вхождений   точки  $z\in \CC$ в $Z$ и  
пишем $z\in Z$,  если $Z(z)>0$. 
Распределение точек  $Z$ можно трактовать  и как меру со значениями в  $\overline \NN_0$ с тем же обозначением  
\begin{equation}\label{Z}
Z(S):=\sum_{z\in S} Z(z)\in \overline \NN_0\quad\text{\it для любого  $S\subset \CC$.}
\end{equation}
Распределение точек $Z$ на   $\CC$ \textit{локально конечно,\/}  если её \textit{считающая радиальная функция} 
$Z^{\rad}(r):=Z\bigl(\overline D(r)\bigr)$ конечна, т.е.  $Z^{\rad}(r)<+\infty$
для любого   $r\in \RR^+$.

Для компакта $K$ в $\CC$  через $C(K)$ обозначаем \textit{банахово пространство непрерывных функций $f\colon K\to \CC$ с\/  $\sup$-нормой}  $\|f\|_{C(K)}:=\sup\Bigl\{\bigl|f(z)\bigr|\Bigm| z\in K \Bigr\}$.
Для открытого подмножества $O\subset \CC$ через $\Hol(O)$ обозначаем \textit{пространство голоморфных функций $f\colon O\to \CC$ с топологией равномерной сходимости на всех компактах $K\subset O$,\/} определяемой $\sup$-полунормами $\|f\|_{C(K)}$. Для компакта $K\subset \CC$ с \textit{внутренностью\/} $\intr K$ через $C(K)\bigcap \Hol(\intr K)$  обозначаем \textit{банахово пространство  непрерывных на $K$ и голоморфных на внутренности $\intr K$ функций $f\colon K\to \CC$ с $\sup$-нормой\/} $\|f\|_{C(K)}$.  Очевидно,  если $\intr K=\emptyset$ --- \textit{пустое множество,\/} то $C(K)\bigcap \Hol(\intr K)=C(K)$.

Система векторов из  топологического векторного пространства \textit{полна\/} в нём,  
если замыкание линейной оболочки этой системы совпадает с этим пространством. Для распределения точек $Z$ на $\CC$ далее рассматривается  полнота лишь  \textit{экспоненциальных систем}
 \begin{equation}\label{Exp}
\Exp^Z:=\Bigl\{ w\underset{w\in \CC}{\longmapsto}w^p\exp (zw)\Bigm| z\in Z, \; Z(z)-1\geq p\in \NN_0\Bigr\}
\end{equation}
\textit{ с распределением показателей $Z$.} \textit{Всюду далее\/} рассматриваются только  системы \eqref{Exp} с \textit{локально конечным распределением показателей} $Z$, поскольку в противном случае система $\Exp^Z$ заведомо полна в любом из рассматриваемых в этой статье функциональных пространств.  

\subsection{Предшествующие результаты}
Детальный обзор по полноте экспоненциальных систем по состоянию  до 2012 г. изложен  в монографии-обзоре \cite{Khsur} первого из авторов. Следующий давно известный результат \cite[гл. IV, \S~1]{Levin56}, \cite[комментарий после теоремы 3.3.5]{Khsur} даёт, по-видимому, самое первое условие полноты экспоненциальной  системы \eqref{Exp} 
в терминах \textit{периметра.}
\begin{thA} 
Если $S\neq \emptyset$ --- ограниченная выпуклая область в $\CC$ и 
 \begin{equation}\label{NG}
\limsup_{r\to +\infty}\frac{1}{r}
\int_{1}^r \frac{Z^{\rad}(t)}{t}\dd t
\geqslant \frac{1}{2\pi} \prm (\partial S),
\end{equation}
то экспоненциальная система $\Exp^Z$ из \eqref{Exp} полна в пространстве $\Hol (S)$.
\end{thA}
Как отмечено в \cite[п.~3.4.1]{Khsur}, если $S\neq \emptyset$ --- \textit{выпуклый компакт} в $\CC$ и нестрогое неравенство $\geqslant$ в \eqref{NG} заменить на строгое неравенство $>$, то система $\Exp^Z$  полна в пространстве 
$C(S)\bigcap \Hol (\intr S)$.

Частными проявлениями \cite[\S~7, п.~4, теорема единственности]{Kha91}, 
\cite[теорема 4.1]{Kha91_1},  \cite[теорема A]{Kha99}, \cite[теорема 3.3.5 и п.~3.4.1]{Khsur} является 
\begin{thB} Если $S\neq \emptyset$ --- ограниченная выпуклая область в $\CC$ и  выполнено хотя бы одно из следующих трёх утверждений: 
\begin{enumerate}[{\rm 1)}] 
\item\label{p1}   для некоторого $p\in [0,1)$    выполнено неравенство 
	\begin{equation}\label{p1n}
	\limsup_{r\to +\infty} \frac{1}{2r}
\int_1^r  \left({\left(\frac{r}{t} \right)}^p  +
{\left( \frac{t}{r} \right)}^p\right)
\frac{Z^{\rad} (t)}{t} \dd t 
\geq \frac1{2\pi}\frac1{1-p^2} \prm (\partial S),
\end{equation}
\item\label{p2} выполнено  неравенство
	\begin{equation}\label{p2n}
\limsup_{1<a\to +\infty} \frac1{\ln a} 
\limsup_{r\to +\infty}
 \int_r^{ar} \frac{Z^{\rad} (t)}{t^2} \dd t  \geq
\frac1{2\pi} \prm (\partial S),	
\end{equation}
\item\label{p3}  для некоторого числа $p > 1$ выполнено неравенство 
\begin{equation}\label{p3n}
\limsup_{r\to +\infty} \frac{1}{2r}
\Biggl( \int_1^r \left(2
-\Bigl(\frac{t}{r}\Bigr)^{p}\right)Z^{\rad} (t)\frac{\dd t}{t} 
+\int\limits_r^{+\infty} {\left({\frac{r}{t}}\right)}^{p}
 {Z^{\rad} (t)}\frac{\dd t }{t} \Biggr)
\geq  \frac{1}{2\pi}\frac{p^2}{p^2 -1}\prm (\partial S),  
\end{equation}
\end{enumerate}
то экспоненциальная система\/ $\Exp^{Z}$ из\/ \eqref{Exp} полна в пространстве\/ $\Hol(S)$. 

Если $S\neq \emptyset $ --- выпуклый компакт в $\CC$ и хотя бы в одном из трёх утверждений\/ 
{\rm \ref{p1}), \ref{p2})}  или {\rm \ref{p3})}
соответствующее неравенство 
\eqref{p1n}, \eqref{p2n} или \eqref{p3n}
выполнено со строгим неравенством $>$ вместо нестрогого неравенства $\geq$, 
 то система $\Exp^Z$  полна в пространстве 
$C(S)\bigcap \Hol (\intr S)$.
\end{thB}  
\begin{remark} При выборе $p=0$ в условии полноты  \eqref{p1n} из части \ref{p1}), а также 
при $p\to +\infty$   в  условии  \eqref{p3n}  части \ref{p3}) теоремы B
 получаем в точности условие полноты \eqref{NG} из теоремы A.
\end{remark}
\begin{remark}
Внешний верхний предел $\limsup\limits_{a\to +\infty}$ можно заменить как на точную нижнюю грань 
$\inf\limits_{a>1}$, так и на  предел $\lim\limits_{a\to +\infty}$,  который существует и все эти три величины совпадают
\cite[теорема 1]{KKh00}, \cite[предложение 6]{SalKha20}, когда конечна \textit{верхняя плотность распределение точек $Z$}
\begin{equation}\label{dens}
\overline\dens (Z) :=\limsup_{r\to +\infty} \frac{Z^{\rad}(r)}{r}\in \overline\RR^+.
\end{equation}
\end{remark}
\subsection{Основной  результат о полноте экспоненциальной системы}
Мы развиваем условие полноты системы $\Exp^{Z}$, 
выраженное неравенством  \eqref{p2n} из утверждения {\rm  \ref{p2})} теоремы B. 

Функция $f\colon X\to \overline \RR$ \textit{положительная\/} и пишем $f\geqslant 0$ \textit{на} $X$,
 если $f(X)\subset \overline \RR^+$,   и \textit{отрицательная\/} и пишем $f\leqslant 0$ \textit{на} $X$,  если противоположная ей функция $-f$  положительная.  Та же функция $f$  
\textit{строго\/} положительная, если  $f(X)\subset \overline \RR^+\setminus 0$ и \textit{строго\/} отрицательная, если
противоположная функция $-f$ строго положительная.  
Функция $f$ 
\textit{возрастающая\/}  (соответственно \textit{строго возрастающая}) на 
\textit{интервале} $X\subset \overline \RR$,  если для  любых $x_1,x_2\in X$ из $x_1<x_2$ следует нестрогое неравенство $f(x_1)\leqslant f(x_2)$ (соответственно строгое неравенство $f(x_1)< f(x_2)$)
Функция $f\colon I\to \overline \RR$ 
\textit{убывающая\/}  (соответственно \textit{строго убывающая}) на $I$, если 
противоположная функция $-f$ возрастающая  (соответственно строго возрастающая) на $I$.  
 
\begin{thm}\label{mthm}
Пусть   $r_0\in \RR^+$,  
$f\colon [r_0, +\infty) \to \RR$ --- выпуклая положительная убывающая функция, $Z$ --- распределение точек  на $\CC$
и $P> 0$ --- строго положительное число. Тогда имеют место  следующие  условия полноты:
\begin{enumerate}[{\rm I.}]
\item\label{PI} Если  выполнено  соотношение  
\begin{equation}\label{sZP}
\sup_{r_0\leqslant r< R<+\infty}
\left(\int_{r<t\leq R} \frac{f\bigl(t^2\bigr)}{t}\dd Z^{\rad}(t)
- \frac{P}{2\pi}\int_r^R\frac{f\bigl(t^2\bigr)}{t}\dd t\right)=+\infty, 
\end{equation}
то для любого компакта $S\subset \CC$ со связным дополнением $\CC\setminus K$ и  с периметром  границы выпуклой оболочки, удовлетворяющим неравенству 
\begin{equation}\label{prP}
\prm (\partial \conv S)\leqslant P ,
\end{equation}
экспоненциальная система $\Exp^Z$ полна в пространстве $C(S)\bigcap \Hol (\intr S)$. 

\item\label{PII} Если выполнено одно из следующих двух условий 
\begin{enumerate}[{\rm 1)}]
\item\label{PII1} расходится интеграл 
\begin{equation}\label{if}
\int_{r_0}^{+\infty}\frac{f\bigl(t^2\bigr)}{t}\dd t=+\infty
\end{equation}
и выполнено соотношение 
\begin{equation}\label{sZPR}
  \limsup_{r_0\leq r\to +\infty} \limsup_{1<a\to +\infty}
\frac{1}{\displaystyle \int_r^{ar}\frac{f\bigl(t^2\bigr)}{t}\dd t}
\int_{r<t\leq ar} \frac{f\bigl(t^2\bigr)}{t}\dd Z^{\rad}(t)
\geqslant \frac{P}{2\pi}, 
\end{equation}

\item\label{PII2}
  бесконечен  двойной нижний предел
\begin{equation}\label{lni}
\liminf_{1<a\to +\infty} \liminf_{r_0\leq r\to +\infty}\int_r^{ar} \frac{f\bigl(t^2\bigr)}{t}\dd t=+\infty
\end{equation} 
и выполнено соотношение 
\begin{equation}\label{sZPRa}
  \limsup_{1<a\to +\infty} \limsup_{r_0\leq r\to +\infty}
\frac{1}{\displaystyle \int_r^{ar}\frac{f\bigl(t^2\bigr)}{t}\dd t}
\int_{r<t\leq ar} \frac{f\bigl(t^2\bigr)}{t}\dd Z^{\rad}(t)
\geqslant \frac{P}{2\pi}, 
\end{equation}
\end{enumerate}
то  для любой  односвязной ограниченной области   $S\subset \CC$, с периметром границы её выпуклой оболочки, удовлетворяющей  неравенству \eqref{prP}, 
система $\Exp^Z$ полна в пространстве $ \Hol (S)$.
\end{enumerate}
\end{thm}

Доказательство теоремы \ref{mthm} будет дано в конце статьи в разделе \ref{Sec:th} после формулировки и доказательства  теоремы \ref{thmuS}  об интегральных оценках  распределений масс Рисса субгармонических функций 
с ограничениями на их рост через опорную функцию множества. Доказательство теоремы \ref{thmuS}
использует построенные во вспомогательном разделе \ref{Sec:stc}   специальные радиальные субгармонические функции на $\CC\setminus 0$, которые конструируются  на основе выпуклых убывающих положительных функций $f$ на 
$\RR^+\setminus 0$.

\begin{remark}\label{rem1i} Если 
$\overline\dens (Z) \overset{\eqref{dens}}{=}+\infty$, то система $\Exp^Z$ полна в   $\Hol(S)$ для любого открытого множества $S\subset  \CC$ и в $C(S)\bigcap \Hol (\intr S)$ при  любом компакте $S\subset \CC$. Таким образом, представляет интерес только  случай конечной верхней плотности $\overline\dens (Z)<+\infty$. Тогда, учитывая существование правой производной $f_+'$ для выпуклой на $\RR^+\setminus 0$ функции $f$,   при $r>0$ имеем равенства 
\begin{equation*}
\int_{r<t\leq R} \frac{f\bigl(t^2\bigr)}{t}\dd Z^{\rad}(t)
=
\frac{f\bigl(R^2\bigr)}{R}Z^{\rad}(R)-\frac{f\bigl(r^2\bigr)}{r}Z^{\rad}(r)
-\int_{r<t\leq R} \biggl(\frac{f\bigl(t^2\bigr)}{t}\biggr)_+' Z^{\rad}(t)\dd t.
\end{equation*}
Отсюда в случае ещё и  убывающей функции $f\geqslant 0$ при $\overline\dens (Z) <+\infty$ получаем 
\begin{equation}\label{ZI}
\int_{r<t\leq R} \frac{f\bigl(t^2\bigr)}{t}\dd Z^{\rad}(t)=
\int_{r}^R \frac{Z^{\rad}(t)}{t^2}\bigl(f(t^2)-2tf_+'(t^2)\bigr)\dd t+O(1)
\quad\text{при $r_0\leqslant r<R<+\infty$.}
\end{equation}
Таким образом, каждый интеграл из левой  части  \eqref{ZI}, входящий   в \eqref{sZP}, \eqref{sZPR} и \eqref{sZPRa},  
можно заменить на интеграл из правой части \eqref{ZI}, поскольку добавление постоянных к этим интегралам не влияет на условия  \eqref{sZP}, \eqref{sZPR} и \eqref{sZPRa}.
\end{remark}

\begin{example} Для выпуклой убывающей функции $f(x)\equiv 1$ при $x>0$ 
интеграл из  \eqref{lni} при любом $r>0$ равен $\ln a$ и выполнено   \eqref{lni}, 
а соотношение \eqref{sZPRa} при  учёте замечания \ref{rem1i} с соотношением \eqref{ZI}
--- это  в точности  \eqref{p2n} для $P=\prm (\partial \conv S)$. 
Таким образом, условие полноты \ref{PII2}) из теоремы~\ref{mthm} действительно обобщает условие полноты 
\ref{p2}) из теоремы B.
\end{example}

\begin{example} Функция  $f\colon x \mapsto \dfrac{1}{\ln x}$ выпуклая убывающая и положительная 
на $[e,+\infty)$, а также удовлетворяет условию \eqref{if}, но не условию \eqref{lni}. 
Следовательно,  такая функция даёт новые условия полноты  в форме \ref{PI} и \ref{PII1}).  
То же самое справедливо при любом $n\in \NN$ для  функции 
\begin{equation*}
f\colon x \underset{x\in [r_0, +\infty)}{\longmapsto} \frac{1}{\underset{\text{$n$-кратно}}{\underbrace{\ln\ln \dots \ln}}\, x}, \quad r_0:=\underset{\text{$n$-кратно}}{\underbrace{{{{e^e}^{{}^{\iddots}}}}^\text{\tiny $e$}}}.
\end{equation*}
\end{example}

\section{Одна конструкция субгармонических функций}\label{Sec:stc}

Последовательность функций $f_n\colon X\to \overline \RR$, $n\in \NN$, \textit{возрастающая,\/} если при каждом $n\in \NN$ разность $f_{n+1}-f_n\geqslant 0$ --- положительная функция на $X$, и \textit{убывающая,\/} если
последовательность $(-f_n)_{n\in \NN}$ противоположных функций $-f_n$ возрастающая. 

\subsection{Выпуклые убывающие функции на открытой положительной полуоси}
\begin{proposition}\label{pr:fF}
Если  функция $f\geqslant 0$ на  $\RR^+\setminus 0$ выпуклая и убывающая, то  
\begin{enumerate}[{\rm 1)}]
\item\label{f1} она  непрерывная
с возрастающими левой  $f_{-}'$ и правой  $f_{+}'$ конечными  производными  на $\RR^+\setminus 0$;
\item\label{f2}
существует предел $\lim\limits_{x\to +\infty} f_{\pm}'(x)=0$ и  выполнены неравенства  
$f_{-}'\leqslant  f_{+}'\leqslant 0$ на $\RR^+\setminus 0$, а также 
\begin{equation}\label{fx2}
0\geqslant f_{\pm}'(x_2)\geqslant\frac{f(x_2)-f(x_1)}{x_2-x_1}\geqslant f_{\pm}'(x_1)\quad \text{при всех $x_2>x_1>0$};
\end{equation}
\item\label{f3} существует убывающая последовательность $(f_n)_{n\in \NN}$ дважды непрерывно дифференцируемых выпуклых убывающих функций $f_n$, которая равномерно стремится к $f$, а также 
\begin{equation}\label{f'r}
0\geqslant f_{\pm}'(x)\geqslant f_n'(x-1/n) \quad\text{при  всех $x> 1/n$}.
\end{equation}
\end{enumerate}
\end{proposition}
\begin{proof} Свойства из \ref{f1})  относятся к элементарным свойствам выпуклых функций \cite[гл.~1, \S~4]{Bou65}, \cite[гл. I]{Hor94}. Свойства  \ref{f2})  легко следуют из убывания выпуклой функции $f\geqslant 0$,
где  для \eqref{fx2} используем, к примеру,   геометрический смысл  левой/правой  производной как соответственно  левой/правой полукасательной и секущей для графика выпуклой убывающей функции. 

Некоторого обсуждения требует, по-видимому, свойство   \ref{f3}), которое может быть получено из методов сглаживания выпуклых функций из книги \cite[Часть~2, гл.~3]{Bra05}. Но здесь  проще схематически описать   возможную конструкцию  требуемой убывающей последовательности   $(f_n)_{n\in \NN}$. 

Функция $f\geq 0$ убывающая и непрерывная на $(0,+\infty)$, поэтому при $n=1$ можно выбрать двустороннюю  последовательность $(x_k)_{k\in \ZZ}$ точек $x_k\in (0,+\infty)$, строго возрастающую в том смысле, что 
$x_k< x_{k+1}$ при любом $k\in \ZZ$, для которой 
$x_k\to +\infty$ при $k\to +\infty$ и $x_k\to 0$ при $k\to -\infty$, а также 
одновременно $0< x_{k+1}-x_{k}\leqslant 1/2$ и $0\leqslant f(x_{k})-f(x_{k+1})\leqslant 1/2$  при каждом 
$k\in \ZZ$. Рассмотрим кусочно-аффинную функцию $l_1$, график которой образован отрезками, соединяющими 
пару точек с координатами $\bigl(x_k, f(x_k)\bigr)$ и $\bigl(x_{k+1}, f(x_{k+1})\bigr)$. По построению 
из выпуклости и убывания $f$ следует, что функция $l_1$ выпуклая, убывающая, удовлетворяет   
неравенству $0\leqslant l_1(x)-f(x)\leqslant 1/2$ при всех $x\in (0,+\infty)$, а также согласно   \eqref{fx2} 
неравенствам $0\geqslant f_{\pm}'(x)\geqslant (l_1)_{\pm}'(x-1/2)$ при всех $x>1/2$. В достаточно  малых окрестностях точек излома графика функции $l_1$ можно сгладить её  выпуклыми сплайнами до 
дважды непрерывно дифференцируемой выпуклой убывающей функции  $f_1\geqslant l_1$ так,  что 
\begin{equation*}
0\leqslant f_1(x)-f(x)\leqslant 1 \quad\text{при всех $x\in (0,+\infty)$}, \quad 0\geqslant f_{\pm}'(x)\geqslant f_1'(x-1) \quad\text{при всех $x>1$.}
\end{equation*}
Для построения функции $f_2$ сначала добавим в каждом отрезке  $[x_k,x_{k+1 }]$ конечное число 
попарно различных точек, начиная с $x_k$ и заканчивая $x_{k+1 }$, так, 
что как расстояние между соседними точками, так и между значениями функции $f$ в этих точках 
 было  $\leqslant 1/4$. За полученной таким образом новой двусторонней  строго возрастающей  последовательностью  
сохраним то же обозначение  $(x_k)_{k\in \ZZ}$.   Снова рассмотрим кусочно-аффинную функцию $l_1$, график которой 
образован отрезками, соединяющими пару точек с координатами $\bigl(x_k, f(x_k)\bigr)$ и 
$\bigl(x_{k+1}, f(x_{k+1})\bigr)$.  По построению из выпуклости и убывания $f$ следует, что функция $l_2$ выпуклая, убывающая, удовлетворяет   неравенствам $f\leq l_2\leq f_1$ на $\RR^+\setminus 0$ и  $0\leqslant l_2(x)-f(x)\leqslant 1/4$ при всех $x\in (0,+\infty)$, а также согласно   \eqref{fx2} 
неравенствам $0\geqslant f_{\pm}'(x)\geqslant (l_2)_{\pm}'(x-1/4)$ при всех $x>1/4$. В очень малых окрестностях точек излома графика функции $l_1$ можем сгладить её выпуклыми сплайнами до 
дважды непрерывно дифференцируемой выпуклой убывающей функции $f_2\geqslant l_2$ так,  что
 $f_2\leqslant f_1$ и 
\begin{equation*}
0\leqslant f_2(x)-f(x)\leqslant 1/2 \quad\text{при всех $x\in (0,+\infty)$}, \quad 0\geqslant f_{\pm}'(x)\geqslant f_2'(x-1/2) \quad\text{при всех $x>1/2$.}
\end{equation*}
Продолжая эту процедуру, на каждом $n$-ом  шаге получаем требуемую выпуклую убывающую дважды непрерывно дифференцируемую функцию $f_n\leq f_{n-1}$ на $\RR^+\setminus 0$ при $n>1$, для которой 
\begin{equation*}
0\leqslant f_n(x)-f(x)\leqslant 1/n \quad\text{при всех $x\in (0,+\infty)$}, \quad 0\geqslant f_{\pm}'(x)\geqslant f_n'(x-1/n) \quad\text{при всех $x>1/n$.}
\end{equation*}
Это завершает доказательство свойства \ref{f3}) с неравенством  \eqref{f'r}. 
\end{proof}
\begin{remark}  Условие убывания функции $f$ на $\RR^+\setminus 0$ в предложении \ref{pr:fF}
можно заменить на формально более слабое условие  
$\limsup\limits_{x\to +\infty}{f(x)}/{x}=0$, поскольку оно  при  выпуклости функции $f\geqslant  0$ на  $\RR^+\setminus 0$
влечёт за собой убывание функции $f$ на $\RR^+\setminus 0$.
\end{remark}

Для  $a\in \overline \RR$ или функции $a$ со значениями в $\overline \RR$
полагаем  $a^+:=\sup\{a,0\}$.

\begin{proposition}\label{pr:fh}
Пусть  функция $f\geqslant 0$   на  $\RR^+\setminus 0$ выпуклая и убывающая.  Тогда   
для любого числа  $R>0$   положительная функция $F_R$, определённая равенствами 
\begin{equation}\label{fF}
F_R(x):=  \biggl(\frac{1}{x}f(x^2)-\frac{1}{R}f(R^2)\biggr)^+\quad \text{при каждом $x\in \RR^+\setminus 0$}, 
\end{equation} 
непрерывная, убывающая и обладает левой $(F_R)_{-}'$ и правой  $(F_R)_{+}'$ конечными  производными всюду на $\RR^+\setminus 0$, а также для неё  выполнены соотношения 
\begin{subequations}\label{F}
\begin{align}
F_R(x)&\geqslant 0=F_R(R)\text{ при $x\in (0,R]$}, 
\tag{\ref{F}+} \label{FR0}
\\
  F_R(x)&\equiv F_R'(x)\equiv  0\text{ при $x\in (R, +\infty)$},
\tag{\ref{F}$_0$} \label{FR1}
\\
F_R(x)&=\frac{1}{x}f(x^2)-\frac{1}{R}f(R^2)\geqslant 0\quad\text{при $x\in (0,R)$}, 
\tag{\ref{F}$_R$}\label{F_R}
\\
(F_R)_{\pm}'(x)&=-\frac{1}{x^2}f(x^2)+2f_{\pm}'(x^2)\quad\text{при $x\in (0,R)$}, 
\tag{\ref{F}$'$}\label{FRpm}
\\
 (F_R)_{-}'(R)&=-\frac{1}{R^2}f(R^2)+2f_{-}'(R^2), \quad (F_R)_{+}'(R)=0, 
\tag{\ref{F}R}\label{FRR}
\\
\bigl|(F_R)_{\pm}'(x)\bigr|&\leq \frac{f(x^2)}{x^2}+\frac{2f(x)}{x(x-1)}
\leq \frac{3f(x)}{x(x-1)}\quad\text{при $x>1$}. 
\tag{\ref{F}$\leqslant$}\label{F1R}
\end{align}
\end{subequations}
\end{proposition}
\begin{proof}  Все свойства функции $F_R$ до группы соотношений \eqref{F} 
автоматический следуют из свойств функции $f$, отражённых  в п. \ref{f1}) предложения \ref{pr:fF}, если учесть что функция-множитель  $x\mapsto 1/x$ при $f$ в \eqref{fF}  убывающая и бесконечно дифференцируемая на $\RR^+\setminus 0$. 

Соотношения \eqref{FR0} очевидно по построению \eqref{fF}  функции $F_R$. 

В силу убывания функции $x\mapsto f(x^2)/x$, во-первых,   имеем   
\begin{equation*}
\frac{1}{x}f(x^2)\leqslant \frac{1}{R}f(R^2)\text{ при $x\in [R, +\infty)$}, 
\end{equation*}
откуда по построению   функции $F_R$ в \eqref{fF} получаем тождества  \eqref{FR1}, 
а во-вторых, 
\begin{equation*}
\frac{1}{x}f(x^2)\geqslant \frac{1}{R}f(R^2)\text{ при $x\in (0,R]$}, 
\end{equation*}
что по построению   функции $F_R$ в \eqref{fF} влечёт за собой  \eqref{F_R}. 
 
Вычисление левой и правой производных для \eqref{F_R}
даёт \eqref{FRpm}, а вместе с \eqref{FR0}--\eqref{FR1} и  \eqref{FRR}.

Наконец, из \eqref{FRpm}, \eqref{FRR} и \eqref{FR1} имеем 
\begin{equation*}
\bigl|(F_R)_{\pm}'(x)\bigr|\leq \frac{f(x^2)}{x^2}
+2\bigl|f_{\pm}'(x^2)\bigr|,
\end{equation*}
а применение неравенства \eqref{fx2} при $x_2:=x^2>x=:x_1>1$ влечёт за собой неравенства 
\begin{equation*}
\bigl|f_{\pm}'(x^2)\bigr|\leqslant \frac{f(x)-f(x^2)}{x^2-x}\leqslant \frac{f(x)}{x(x-1)}.
\end{equation*} 
Таким образом, установлено первое неравенство в \eqref{F1R}. Второе неравенство в \eqref{F1R}при $x>1$
для убывающей функции $f\geqslant 0$ очевидно. 
\end{proof}

\subsection{Построение субгармонических функций с помощью  выпуклых}

\begin{proposition}\label{pr:cbF}
Для любой функции $F_R$ вида \eqref{fF} из предложения\/ {\rm \ref{pr:fh}} с выпуклой и убывающей  
функцией $f\geqslant 0$ на  $\RR^+\setminus 0$  радиальная функция
\begin{equation}\label{ap}
V(re^{i\theta}):= F_R(r)\overset{\eqref{fF}}{=}
 \biggl(\frac{1}{r}f(r^2)-\frac{1}{R}f(R^2)\biggr)^+, \quad 0<r<+\infty,   0\leq \theta < 2\pi, \quad re^{i\theta}\in \CC,
\end{equation}
  --- положительная непрерывная субгармоническая функция на $\CC\setminus 0$ и 
существует последовательность функций $(f_n)_{n\in \NN}$ со всеми свойствами из утверждения\/ {\rm \ref{f3})} 
предложения {\rm \ref{pr:fF}}, для которых последовательность  субгармонических радиальных на $\CC\setminus 0$ функций $V_{n}$ вида 
\begin{equation}\label{apn}
V_{n}(re^{i\theta}):= \biggl(\frac{1}{r}f_n(r^2)-\frac{1}{R}f_n(R^2)\biggr)^+, \quad 0<r<+\infty,  \quad 0\leq \theta < 2\pi, 
\end{equation}
 сходится  к функции $V$ равномерно на компактах из $\CC\setminus 0$, а каждая из функций $V_n$ --- это сужение на 
проколотый круг $\overline D(R)\setminus 0$ дважды непрерывно дифференцируемой на $\CC\setminus 0$ функции 
\begin{equation}\label{without+}
v_n\colon re^{i\theta}\longmapsto \frac{1}{r}f_n(r^2)-\frac{1}{R}f_n(R^2), \quad re^{i\theta}\in \CC\setminus 0.
\end{equation}
При этом  модуль производной по внешней нормали $\vec n_{\operatorname{out}}$
к границе  кольца $D(R)\setminus \overline D(r)$ в точках на окружности $\partial \overline D(r)\subset \partial \bigl(D(R)\setminus \overline D(r)\bigr)$ для каждой из функций 
 $V_n$ и $V$ оценивается как
\begin{equation}\label{drob}
\sup_{\theta\in \RR}\max\Biggl\{ \biggl|\frac{\partial V}{\partial \vec n_{\operatorname{out}}}(re^{i\theta})\biggr|, 
\sup_{n\in \NN} \biggl|\frac{\partial V_n}{\partial \vec n_{\operatorname{out}}}(re^{i\theta})\biggr|
\Biggr\}
\leqslant 
\frac{3f_1(r)}{r(r-1)}, \quad 1<r< R.
\end{equation}

\end{proposition}
\begin{proof} 
По построению \eqref{apn} и по свойствам функций $f_n$ каждая из функций  $V_{n}$ --- это сужение дважды непрерывно дифференцируемой на $\CC\setminus 0$ функции  \eqref{without+}.
Из равномерной сходимости последовательности $(f_n)_{n\in \NN}$ к функции $f$ и вида радиальных  функций $V$ и $V_n$ в \eqref{ap}--\eqref{apn} следует, что для любого $r\in (0,R)$
последовательность  функций $V_n$ равномерно сходится к $V$ на кольце $\overline D(R)\setminus D(r)$, а значит сходится равномерно на компактах из $\CC\setminus 0$.  

Сначала установим субгармоничность на $\CC\setminus 0$ функций $v_n$ из \eqref{without+},
для которых 
\begin{equation}\label{f'}
\frac{\partial v_n}{\partial r}(re^{i\theta})\overset{\eqref{FRpm}}{=}-\frac{1}{r^2}f_n(r^2)+2f_n'(r^2),\quad 
\frac{\partial^2 v_n}{\partial r^2}(re^{i\theta})=\frac{2}{r^3}f_n(r^2)-\frac{2}{r}f'_n(r^2)+4rf_n''(r^2), \quad r>0.
\end{equation}
Прямое вычисление оператора Лапласа в полярных координатах 
\begin{equation*}
{\bigtriangleup}=\frac{\partial^2}{\partial r^2}+\frac1{r}\frac{\partial}{\partial r}
+\frac{1}{r^2}\frac{\partial^2}{\partial\theta^2}
\end{equation*}  
от  функции $v_n$ из \eqref{without+} на  $\CC\setminus 0$
даёт согласно \eqref{f'} равенство 
\begin{equation*}
{\bigtriangleup}v_n(re^{i\theta})\overset{\eqref{f'}}{=}\Bigl(\frac{2}{r^3}f_n(r^2)-\frac{2}{r}f'_n(r^2)+4rf_n''(r^2)\Bigr)
+\frac{1}{r} \Bigl(-\frac{1}{r^2}f_n(r^2)+2f_n'(r^2)\Bigr)
\end{equation*}
откуда, после раскрытия скобок и приведения подобных, получаем  
\begin{equation}\label{dfh}
{\bigtriangleup}v_n(re^{i\theta})
=\frac{1}{r^3}f_n(r^2)+4rf_n''(r^2)\geq 0\text{ при  
$re^{i\theta}\in \CC\setminus 0$} 
\end{equation}
в силу положительности $f_n\geqslant 0$ и выпуклости $f_n$, дающей  положительность $f_n''\geqslant 0$ на $\RR^+\setminus 0$. Таким образом, каждая функция $v_n$ из  \eqref{f'} субгармоническая. 
Следовательно, и каждая функция $V_n=v_n^+:=\sup\{v_n,0\}$ по построению  \eqref{apn} субгармоническая на $\CC\setminus 0$. Равномерная сходимость на компактах из $\CC\setminus 0$  последовательности субгармонических непрерывных функций $V_n$ к функции $V$ обеспечивает и её субгармоничность на $\CC\setminus 0$. 

По неравенству \eqref{F1R} предложения \ref{pr:fh} модули левой производной  по радиусу для первых сомножителей в определениях \eqref{ap} для $V$ и  \eqref{apn} для $V_n$ 
не превышают соответственно дробей 
\begin{equation*}
\frac{3f(r)}{r(r-1)}, \quad \frac{3f_n(r)}{r(r-1)},\quad 1<r<R, 
\end{equation*}
каждая из которых не больше 
\begin{equation*}
\frac{3f_1(r)}{r(r-1)},\quad 1<r<R.  
\end{equation*}
Отсюда для модуля производной $\frac{\partial}{\partial \vec n_{\operatorname{out}}}$
к кольцу $D(R)\setminus \overline D(r)$ на  $\partial \overline D(r)$ от  функций 
$V$ и    $V_n$,   равного  модулю левой производной  по радиусу в точках на $\partial \overline D(r)$, 
получаем \eqref{drob}.  
\end{proof}

\section{Оценки распределений масс Рисса субгармонических функций}\label{Sec:th}

Для $z\in \CC$, как обычно, $\Bar z$ --- комплексное число, сопряжённое с $z$. 
\textit{Опорной функцией множества\/}  $S\subset \CC$ называется  функция \cite{BF}
\begin{equation}\label{Spf}
{\Spf}_{S}\colon z\underset{z\in \CC}{\longmapsto} \sup\limits_{s\in S}\operatorname{Re} s\Bar z\in \overline \RR. 
\end{equation}
Опорная функция ${\Spf}_{S}$ принимает только конечные значения из $\RR$, если и только если  $S\neq \emptyset$ --- ограниченное в $\CC$.  Для такого $S\subset \CC$  его опорная функция по построению \eqref{Spf} 
\textit{положительно однородная} и \textit{выпуклая,\/}  а следовательно, \textit{непрерывная\/} и \textit{субгармоническая  на\/} $\CC$ \cite{HK}, \cite{Rans}.

Каждой субгармонической на $\CC$ функцией $u\not\equiv -\infty$ 
соответствует   положительная  мера Радона $\varDelta_u:=\frac{1}{2\pi}{\bigtriangleup} u$, 
где ${\bigtriangleup}$  --- {\it оператор Лапласа,\/} действующий в смысле теории обобщённых функций 
 \cite{HK}, \cite{Hor94}, \cite{Rans}.  Меру $\varDelta_u$ называем \textit{распределением масс Рисса\/}
субгармонической функции $u$.  

Для меры Радона $\mu$ на $\CC$ её  \textit{считающая радиальная   функция\/}  обозначается и определяется как
\begin{equation*}
\mu^{\rad}\colon r\underset{r\in \RR^+}{\longmapsto} \mu\bigl(\overline D(r)\bigr).
\end{equation*}

\begin{thm}\label{thmuS} Пусть $S\subset \CC$ --- ограниченный компакт в $\CC$  с опорной функцией $\Spf_{S}$ и периметром $\prm (\partial \conv S)>0$ границы $\partial \conv S$ выпуклой оболочки $\conv S$ множества $S$.   Если $u\not\equiv -\infty$ --- субгармоническая функция с распределением масс Рисса $\varDelta_u$,  удовлетворяющая неравенствам
\begin{equation}\label{uSpf}
u(z)\leqslant  \Spf_{S}(z)+c\quad\text{для некоторого $c\in \RR$ при всех $z\in \CC$},
\end{equation}  
то для любой положительной  выпуклой убывающей на 
$\RR^+\setminus 0$  функции  $f$  существует число $C\in \RR$, для которого выполнено неравенство  
\begin{equation}\label{ituD}
\int_{r}^R  \frac{f\bigl(t^2\bigr)}{t}\dd \varDelta_u^{\rad}(t)
\leqslant \frac{1}{2\pi}\prm (\partial \conv S)\int_r^R\frac{f\bigl(t^2\bigr)}{t}\dd t+C
\quad\text{при всех\/ $1<r<R<+\infty$.}
\end{equation}
\end{thm}
\begin{proof} Не умаляя общности, в \eqref{uSpf} можно положить  $c=0$, поскольку распределение масс 
Рисса для субгармонической функции $u-c$ то же самое, что и для $u$. 

Основную роль  при доказательстве будет играть следующая 

\begin{lemma}[{\rm \cite[леммы 2.2--2.3]{Kha91_1}}]\label{lem2_2} 
Пусть $0<r<R<+\infty$ и  функция 
$V$ положительна на замкнутом кольце $\overline D(R)\setminus D(r)$,
 субгармоническая в его  внутренности $D(R)\setminus \overline D(r)$, 
тождественно равна  нулю на окружности $\partial \overline D(R)$ и совпадает с  сужением  
на   $\overline D(R)\setminus D(r)$
некоторой дважды непрерывно дифференцируемой  в окрестности  кольца
 $\overline D(R)\setminus D(r)$ функции. 
Используя инверсию  функции $V$ относительно окружности $\partial \overline D(r)$, построим  положительную 
на\/ $\CC$ функцию 
\begin{equation}\label{V*}
V^*(z):=
\begin{cases}
V(z)&\text{при $r< |z|\leqslant R$},\\
V(r^2/\Bar z)&\text{при $r^2/R< |z|\leqslant  r$},\\
0&\text{при  $|z|\leqslant  r^2/R$ и $|z|>R$},
\end{cases} \quad z\in \CC.
\end{equation}

Тогда для любой пары субгармонических на окрестности круга $\overline D(R)$ функций $u\not\equiv -\infty$ и $M\not\equiv -\infty$ с распределениями масс Рисса соответственно $\varDelta_u$ и $\varDelta_M$
из неравенства   $u\leq M$ на этой окрестности следует неравенство 
\begin{equation}\label{VuMn}
\int_{\CC} V^*\dd \varDelta_u\leqslant 
\int_{\CC} V^*\dd \varDelta_M+ \frac{r}{\pi}\int_0^{2\pi} \bigl(u(re^{i\theta})-M(re^{i\theta})\bigr)
\frac{\partial V}{\partial \vec n_{\operatorname{out}}}(re^{i\theta}) \dd \theta, 
\end{equation}  
где  по-прежнему, как и в конце формулировки предложения {\rm \ref{pr:cbF}},
$\frac{\partial}{\partial \vec n_{\operatorname{out}}}$  --- оператор дифференцирования по внешней нормали  
к кольцу $D(R)\setminus \overline D(r)$ в точках на окружности $\overline D(r)$.
\end{lemma}

В качестве функции $M$ выберем опорную функцию $\Spf_S$.
Её распределение масс  Рисса в полярных координатах определяется как произведение мер \cite[п. 3.3.1]{Khsur}
\begin{equation}\label{rl}
\dd \varDelta_M=\dd \varDelta_{\Spf_S}=\frac{1}{2\pi}\dd r \otimes \dd l_{\conv S}(\theta),
\end{equation}
где $l_{\conv S}(\theta)$ --- длина дуги границы $\partial \conv S$, отсчитываемой при движении по границе <<против часовой стрелки>> от последней точки опоры   опорной к $\clos \conv S$ прямой, ортогональной положительной полуоси $\RR^+$, до последней точки опоры   опорной к $\clos \conv S$ прямой, ортогональной направлению радиус-вектора  точки $e^{i\theta}$ \cite{BF}, \cite[п. 3.3.1]{Khsur}. В частности, 
\begin{equation}\label{lprm}
\int_0^{2\pi}\dd l_{\conv S}(\theta)=l_{\conv S}(2\pi)-l_{\conv S}(0)=\prm (\partial \conv S).
\end{equation}

В качестве функции $V$ выбираем  функции $V_n$, определённые как  в \eqref{apn}--\eqref{without+} из предложения   
\ref{pr:cbF} для функции из \eqref{ap} по функции $f$ из условия доказываемой теоремы  \ref{thmuS}, $n\in \NN$.
Для функций $V:=V_n$ по предложению \ref{pr:cbF} выполнены все условия леммы \ref {lem2_2},  
а положительные по построению  функции $V_n^*$, построенные как в \eqref{V*}, будут иметь явный радиальный вид 
\begin{equation}\label{V*n}
0\leq V_n^*(z):=
\begin{cases}
{\displaystyle\frac{1}{|z|}f_n\bigl(|z|^2\bigr)-\frac{1}{R}f_n(R^2)}&\text{при $r< |z|\leqslant R$},\\
{\displaystyle\frac{|z|}{r^2}f_n\bigl(r^4/|z|^2\bigr)-\frac{1}{R}f_n(R^2)}
&\text{при $r^2/R< |z|\leq r$},\\
0&\text{при  $|z|\leq r^2/R$ и $|z|>R$},
\end{cases} \quad z\in \CC.
\end{equation}
По лемме  \ref {lem2_2} из заключения \eqref{VuMn} с учётом \eqref{uSpf} для  $c=0$, \eqref{rl} 
и \eqref{drob} при $1<r<R<+\infty$ следует 
\begin{multline*}
\int_{\CC} V_n^*\dd \varDelta_u\leqslant 
\frac{1}{2\pi}\int_{r^2/R}^{R}\int_0^{2\pi} V_n^*(te^{i\theta})\dd t\dd l_{\conv S}(\theta)+ \frac{r}{\pi}\int_0^{2\pi} \Bigl(\bigl|u(re^{i\theta})|+\bigl|\Spf(re^{i\theta})\bigr|\Bigr)
\biggl|\frac{\partial V_n}{\partial \vec n_{\operatorname{out}}}(re^{i\theta})\biggr| \dd \theta\\
\overset{\eqref{V*n},\eqref{drob}}{\leqslant}  
\left(\int_{r^2/R}^{r} \frac{t}{r^2}f_n\bigl(r^4/t^2\bigr)\dd t
+\int_{r}^{R} \frac{f_n\bigl(t^2\bigr)}{t}\dd t\right)
\frac{1}{2\pi}\int_0^{2\pi}\dd l_{\conv S}(\theta)\\
+ \frac{r}{\pi}\biggl(\int_0^{2\pi} \bigl|u(re^{i\theta})\bigr|
 \dd \theta+\int_0^{2\pi} \bigl|\Spf(re^{i\theta})\bigr|
\dd \theta \biggr)\frac{3f_1(r)}{r(r-1)}.
\end{multline*}  
Отсюда согласно \eqref{lprm} и  \eqref{V*n} при $2\leq r<R<+\infty$ получаем 
\begin{multline}\label{VuMn+}
\int_{\overline D(R)\setminus \overline D(r)} \frac{f_n\bigl(|z|^2\bigr)}{|z|}\dd \varDelta_u(z)
\leqslant \int_{\overline D(R)\setminus \overline D(r)} \frac{1}{R}f_1(R^2)\dd \varDelta_u
\\+\left(\frac{1}{2}f_1(r^2) +\int_{r}^{R} \frac{f_n\bigl(t^2\bigr)}{t}\dd t\right)\frac{1}{2\pi}\prm(\partial \conv S)\\
+ \frac{3f_1(r)}{r}\int_0^{2\pi} \bigl|u(re^{i\theta})\bigr| \dd \theta+2 f_1(r)\max_{\theta\in [0,2\pi)} \bigl|\Spf(e^{i\theta})\bigr|.
\end{multline}
Опорная   функция $\Spf_S$  удовлетворяет ограничению $\Spf(z)=O\bigl(|z|\bigr)$ при $z\to \infty$
в силу её положительной однородности. Отсюда  функция $u$ согласно   \eqref{uSpf} удовлетворяет условию 
\begin{equation*}
\limsup_{z\to +\infty} \frac{u(z)}{|z|}<+\infty,
\end{equation*}
т.е. функция $u$ \textit{конечного типа при порядке $1$.} Следовательно, 
для распределения масс Рисса $\varDelta_u$ такой функции имеем $\varDelta_u^{\rad}(t)=O(t)$, 
а также  \cite[лемма 6.2]{KhaShm19}
\begin{equation*}
\int_0^{2\pi} \bigl|u(re^{i\theta})\bigr| \dd \theta=O(r)\quad \text{при $r\to +\infty$.}
\end{equation*}
Отсюда  ввиду убывания функции $f_1$ из  \eqref{VuMn+} следует существование числа $C\in \RR^+$, 
для которого 
\begin{equation*}
\int_r^R \frac{f_n\bigl(t^2\bigr)}{t}\dd \varDelta_u^{\rad}(t)
\leqslant \frac{1}{2\pi}\prm(\partial \conv S) \int_{r}^{R} \frac{f_n\bigl(t^2\bigr)}{t}\dd t+C
\quad\text{при всех $2\leqslant r<R<+\infty$.}
\end{equation*}
Следовательно, для убывающей и равномерно сходящейся к $f$ последовательности функций $f_n$ из предложения  \ref{pr:fF} имеем неравенства 
\begin{equation*}
\int_r^R \frac{f\bigl(t^2\bigr)}{t}\dd \varDelta_u^{\rad}(t)
\leqslant \frac{1}{2\pi}\prm(\partial \conv S) \int_{r}^{R} \frac{f\bigl(t^2\bigr)}{t}\dd t+C
\quad\text{при всех $2\leqslant r<R<+\infty$,}
\end{equation*}
где нижнее ограничение $2\leqslant r$  ввиду 
конечности интегралов по отрезкам  $[1,2]$  можно заменить на $1\leqslant r$, увеличивая, если  необходимо,  число $C\in \RR^+$.  
\end{proof}

\begin{proof}[Доказательство теоремы\/ {\rm \ref{mthm}}.]
Отметим сначала, что функция $f$ может быть продолжена на весь луч $\RR^+\setminus 0$
как выпуклая, положительная  и убывающая, возможно, изменив $r_0$ на чуть большее. 

Допустим, что экспоненциальная система $\Exp^{Z}$ не полна в пространстве $C(S)\bigcap \Hol (\intr S)$
для некоторого компакта $S\subset \CC$ со связным дополнением $\CC\setminus S$, удовлетворяющего \eqref{prP}. По теореме Хана\,--\,Банаха  \cite[п.~1.1.1]{Khsur}
и теореме  Рисса о  представлении линейных функционалах это означает, что найдётся 
  комплекснозначная мера Радона $\mu\neq 0$ с носителем $\supp \mu \subset S$, для которой ненулевая целая 
функция 
\begin{equation}\label{gei}
g(z)\underset{z\in \CC}{=}\int_{S} e^{zs}\dd\mu(s)\not\equiv 0,
\end{equation}
\textit{обращается в нуль на $Z$ с учётом кратности\/} в том смысле, что для каждой точки $z\in \CC$
кратность корня целой функции $f$ в точке $z$ не меньше $Z(z)$. При этом функция $u:=\ln |g|\not\equiv -\infty$ --- субгармоническая с распределением масс Рисса $\varDelta_u\geqslant Z$, где $Z$  рассматривается как мера, определённая в \eqref{Z}. Из представления \eqref{gei} и определения \eqref{Spf} опорной функции следует 
\begin{equation}\label{uSy}
u(z)=\ln\bigl|g(z)\bigr|\leq \ln \Bigl(\exp \bigl(\sup_{s\in S}\Re zs\bigr)|\mu|(S)\Bigr)
\leqslant \Spf_{\Bar S}(z)+\ln |\mu|(S)\quad\text{при всех $z\in \CC$},
\end{equation}
где $|\mu|$ --- полная вариация меры $\mu$, а $\Bar S:=\bigl\{\Bar z\bigm| z\in S\bigr\}$ --- компакт, зеркально симметричный компакту $S$ относительно  вещественной оси.
Неравенство в \eqref{uSy} означает, что выполнено условие  \eqref{uSpf} теоремы \ref{thmuS}
с $\Bar{S}$ вместо $S$. По теореме \ref{thmuS} для любой положительной  выпуклой убывающей на 
$\RR^+\setminus 0$  функции  $f$  существует число $C\in \RR$, для которого выполнено неравенство  
\eqref{ituD} с $\prm \bigl(\partial \conv \Bar{S}\bigr)$, очевидно, равным   $\prm (\partial \conv {S})$.
Таким образом, ввиду неравенства для мер $Z\leqslant \varDelta_u$ получаем 
\begin{equation}\label{intg}
\int_{r}^R  \frac{f\bigl(t^2\bigr)}{t}\dd Z^{\rad}(t)\leqslant 
\int_{r}^R  \frac{f\bigl(t^2\bigr)}{t}\dd \varDelta_u^{\rad}(t)
\leqslant \frac{1}{2\pi}\prm (\partial \conv S)\int_r^R\frac{f\bigl(t^2\bigr)}{t}\dd t+C
\end{equation}
для некоторого числа $C\in \RR$ при всех $1<r<R<+\infty$. Нижнее ограничение $1<r$ здесь 
можно поменять на $r_0\leqslant r$, поскольку возможно и  добавляющиеся при это интегралы  
по интервалам с концами $r_0$ и $1$ конечны. Неравенство между крайними частями \eqref{intg}
означает, что 
\begin{equation*}
\sup_{r_0\leqslant r< R<+\infty}\left(\int_{r<t\leq R} \frac{f\bigl(t^2\bigr)}{t}\dd Z^{\rad}(t)
- \frac{\prm (\partial \conv S)}{2\pi}\int_r^R\frac{f\bigl(t^2\bigr)}{t}\dd t\right)<+\infty, 
\end{equation*}
откуда ввиду неравенства $\prm (\partial \conv S)\overset{\eqref{prP}}{\leqslant} P $
конечна точная верхняя грань в левой части равенства  \eqref{sZP}, что противоречит  условию   
\eqref{sZP}. Это противоречие доказывает, что на самом деле  
 система $\Exp^Z$ полна в  $C(S)\bigcap \Hol (\intr S)$. 
Таким образом, часть \ref{PI} теоремы \ref{mthm} доказана. 

Перейдём к доказательству части \ref{PII} теоремы \ref{mthm} в условиях сначала из \ref{PII1}).
Пусть теперь $S\neq \emptyset$ --- односвязная  ограниченная область  в $\CC$, удовлетворяющяя  \eqref{prP},
а $K\subset S$ --- непустой компакт со связным дополнением $\CC\setminus K$. Тогда $\conv K$ --- выпуклый компакт в выпуклой открытой области $\conv S$, очевидно, со связным  дополнением $\CC\setminus \conv K$, а  
$\prm(\partial \conv K)<\prm(\partial \conv S)$ \cite{Leicht}. Выберем промежуточные  числа  $P_1, P_2, P_3\in \RR^+$ так,что 
\begin{equation}\label{P123}
\prm(\partial \conv K)\leq P_3<P_2<P_1<\prm(\partial \conv S)\overset{\eqref{prP}}{\leq} P. 
\end{equation}  
  
Соотношение \eqref{sZPR} означает, что  найдётся достаточно большое $r\geqslant r_0$, для которого  
\begin{equation*}
 \limsup_{1<a\to +\infty}
\frac{1}{\displaystyle \int_{r}^{a{r}}\frac{f\bigl(t^2\bigr)}{t}\dd t}
\int_{r<t\leq ar} \frac{f\bigl(t^2\bigr)}{t}\dd Z^{\rad}(t)
\geqslant \frac{P_1}{2\pi}.
\end{equation*}
Последнее согласно \eqref{P123} означает, что найдётся возрастающая последовательность $(a_k)_{k\in \NN}$
чисел  $1<a_k\to \infty$, для которой 
\begin{equation*}
 \frac{1}{\displaystyle \int_{r_n}^{a_k{r}}\frac{f\bigl(t^2\bigr)}{t}\dd t}
\int_{r<t\leq a_kr} \frac{f\bigl(t^2\bigr)}{t}\dd Z^{\rad}(t)
\geqslant \frac{P_2}{2\pi}\overset{\eqref{P123}}{=}\frac{P_3}{2\pi}+\frac{P_2-P_3}{2\pi}
\quad\text{при всех  $k\in \NN$,} 
\end{equation*}
где последнее слагаемое строго положительно. Отсюда следует 
\begin{equation*}
 \int_{r<t\leq a_kr} \frac{f\bigl(t^2\bigr)}{t}\dd Z^{\rad}(t)
- \frac{P_3}{2\pi} \int_{r}^{a_k{r}}\frac{f\bigl(t^2\bigr)}{t}\dd t
\geqslant \frac{P_2-P_3}{2\pi}\int_{r_n}^{a_k{r}}\frac{f\bigl(t^2\bigr)}{t}\dd t
\quad\text{при всех  $k\in \NN$,} 
\end{equation*}
В  силу расходимости интеграла \eqref{if} и строгой положительности $P_2-P_3>0$ правая часть здесь неограниченно возрастает и поэтому  получаем 
\begin{equation*}
\sup_{k}\left(\int_{r<t\leq a_kr} \frac{f\bigl(t^2\bigr)}{t}\dd Z^{\rad}(t)
- \frac{P_3}{2\pi} \int_{r}^{a_k{r}}\frac{f\bigl(t^2\bigr)}{t}\dd t\right)=+\infty. 
\end{equation*}
Тем более, выполнено соотношение \eqref{sZP} с числом $P_3\overset{\eqref{P123}}{\geqslant} \prm(\partial \conv K)$ вместо $P$.  Отсюда по доказанной части \ref{PI} теоремы \ref{mthm}
экспоненциальная система $\Exp^Z$ полна в пространстве $C(K)\bigcap \Hol (\intr K)$. В силу произвола в выборе компакта $K$ со связным дополнением $\CC\setminus K$ в односвязной  ограниченной области  $S\subset \CC$ такие компакты исчерпывают односвязную   ограниченную область  $S\subset \CC$. Это доказывает  
полноту экспоненциальной системы $\Exp^Z$ в пространстве $\Hol (S)$.
Тем самым в условиях \ref{PII1}) часть \ref{PII} теоремы \ref{mthm} доказана. 

Перейдем к доказательству части \ref{PII} теоремы \ref{mthm} в условиях \ref{PII2}).
Пусть, по-прежнему,  $S\neq \emptyset$ --- открытое ограниченное множество в $\CC$, удовлетворяющее  \eqref{prP},
и пусть выполнены условия \eqref{lni}--\eqref{sZPRa}.  
Предположим, что экспоненциальная  система $\Exp^Z$ не полна в пространстве $\Hol (S)$. Это означает, что существует такой компакт $K\subset S$ со связным дополнением $\CC\setminus K$, что эта  система $\Exp^Z$  не полна в 
$C(K)\bigcap \Hol(\intr K)$. Тогда по  доказанной части \ref{PI} теоремы \ref{mthm} имеем 
\begin{equation*}
\sup_{r_0\leqslant r< R<+\infty}\left(\int_{r<t\leq R} \frac{f\bigl(t^2\bigr)}{t}\dd Z^{\rad}(t)
- \frac{\prm(\partial \conv K)}{2\pi}\int_r^R\frac{f\bigl(t^2\bigr)}{t}\dd t\right)<+\infty, 
\end{equation*}
откуда следует существование числа $C\in \RR$, для которого, выбирая $R:=ar$, получаем 
  \begin{equation*}
\int_{r<t\leq ar} \frac{f\bigl(t^2\bigr)}{t}\dd Z^{\rad}(t)\leq  \frac{\prm(\partial \conv K)}{2\pi}
\int_r^{ar}\frac{f\bigl(t^2\bigr)}{t}\dd t +C\text{   при любых $r\geqslant r_0$
и $a>1$.} 
\end{equation*}
Поделив обе части этого неравенства на последний интеграл и переходя к пределу при  $r\to +\infty$
в обеих частях неравенства, приходим к соотношениям 
   \begin{align*}
\limsup_{r_0\leqslant r\to +\infty}\frac{1}{\displaystyle \int_r^{ar}\frac{f\bigl(t^2\bigr)}{t}\dd t}
\int_{r<t\leq ar} \frac{f\bigl(t^2\bigr)}{t}\dd Z^{\rad}(t)&\leqslant \limsup_{r_0\leqslant r\to +\infty}
\left( \frac{\prm(\partial \conv K)}{2\pi}+\frac{C}{\displaystyle \int_r^{ar}\frac{f\bigl(t^2\bigr)}{t}\dd t}\right) 
\\&= \frac{\prm(\partial \conv K)}{2\pi}+\frac{C}{\liminf\limits_{r_0\leqslant r\to +\infty}\displaystyle \int_r^{ar}\frac{f\bigl(t^2\bigr)}{t}\dd t}
\end{align*}
Применяя второй верхний предел по $1<a\to +\infty$ к крайним частям этого соотношения,   получаем 
\begin{equation*}
\limsup_{1<a\to +\infty}\limsup_{r_0\leqslant r\to +\infty}\frac{\displaystyle \int_{r<t\leq ar} \frac{f\bigl(t^2\bigr)}{t}\dd Z^{\rad}(t)}{\displaystyle \int_r^{ar}\frac{f\bigl(t^2\bigr)}{t}\dd t}
\leqslant \frac{\prm(\partial \conv K)}{2\pi}+\frac{C}{\liminf\limits_{1<a\to +\infty}\liminf\limits_{r_0\leqslant r\to +\infty}\displaystyle \int_r^{ar}\frac{f\bigl(t^2\bigr)}{t}\dd t},
\end{equation*}
где по условию \eqref{lni} знаменатель последней дроби даёт  $+\infty$, вследствие чего  
\begin{equation*}
\limsup_{1<a\to +\infty}\limsup_{r_0\leqslant r\to +\infty}\frac{\displaystyle \int_{r<t\leq ar} \frac{f\bigl(t^2\bigr)}{t}\dd Z^{\rad}(t)}{\displaystyle \int_r^{ar}\frac{f\bigl(t^2\bigr)}{t}\dd t}\leqslant \frac{\prm(\partial \conv K)}{2\pi}<\frac{\prm(\partial \conv S)}{2\pi}\overset{\eqref{prP}}{\leq} \frac{P}{2\pi}.
\end{equation*}
Здесь  строгое промежуточное неравенство $<$ порождает противоречие с условием 
\eqref{sZPRa}. Следовательно,   система $\Exp^Z$ полна в пространстве $\Hol (S)$.
Теорема доказана.
\end{proof}

\end{document}